\documentclass[10pt]{amsart}
\usepackage{amsfonts}
\usepackage{ifthen}
\usepackage{amsthm}
\usepackage{amsmath,mathrsfs}
\usepackage{graphicx}
\usepackage{amscd,amssymb,amsthm}
\usepackage{color}
\usepackage{hyperref}

\newcounter{minutes}
\divide\time by 60
\newcounter{hours}
\multiply\time by 60 \addtocounter{minutes}{-\time}

\newtheorem{theorem}{Theorem}

\newtheorem{corollary}{Corollary}

\title[Integral representations for the Fox-Wright functions with applications]{New Integral representations for the Fox-Wright functions and its applications}

\author[K. Mehrez]{Khaled Mehrez}
\address{Khaled Mehrez. D\'epartement de Math\'ematiques, Facult\'e de Sciences de Tunis, Universit\'e Tunis El Manar, Tunisia.}
\address{D\'epartement de Math\'ematiques ISSAT Kasserine, Universit\'e de Kairouan, Tunisia}
\email{k.mehrez@yahoo.fr}



\begin{document}

\def\thefootnote{}
\footnotetext{ \texttt{File:~\jobname .tex,
          printed: \number\year-0\number\month-\number\day,
          \thehours.\ifnum\theminutes<10{0}\fi\theminutes}
} \makeatletter\def\thefootnote{\@arabic\c@footnote}\makeatother

\maketitle

\begin{abstract}
Our aim in this paper is to derive several new integral representations of the Fox-Wright  functions. In particular, we give new Laplace and Stieltjes transform for this special functions  under a special restriction on parameters. From the positivity conditions for the weight in these representations, we found sufficient conditions to be imposed on the parameters of the  Fox-Wright  functions that it be completely monotonic. As applications, we derive a class of function related to the Fox H-functions is positive definite and an investigation of a class of the Fox H-function is non-negative. Moreover,  we extended the Luke's inequalities and we establish a new Tur\'an type inequalities for the Fox-Wright function. Finally,  by appealing to each of the Luke's inequalities, two sets of two-sided bounding inequalities for the generalized Mathieu's type series are proved.
\end{abstract}

\section{ Introduction}
In this paper, we use the Fox-Wright generalized hypergeometric function ${}_p\Psi_q[.]$  with $p$ numerator parameters $\alpha_1,...,\alpha_p$ and $q$ denominator parameters $\beta_1,...,\beta_q,$  which are defined by \cite[p. 4, Eq. (2.4)]{Z}
\begin{equation}\label{11}
{}_p\Psi_q\Big[_{(\beta_1,B_1),...,(\beta_q,B_q)}^{(\alpha_1,A_1),...,(\alpha_p,A_p)}\Big|z \Big]={}_p\Psi_q\Big[_{(\beta_q,B_q)}^{(\alpha_p,A_p)}\Big|z \Big]=\sum_{k=0}^\infty\frac{\prod_{i=1}^p\Gamma(\alpha_l+kA_l)}{\prod_{j=1}^q\Gamma(\beta_l+kB_l)}\frac{z^k}{k!},
\end{equation}
$$\left(\alpha_i,\beta_j\in\mathbb{C},\;\textrm{and}\;\;A_i,B_j\in\mathbb{R}^+\; (i=1,...,p,j=1,...,q)\right),$$
where, as usual,
$$\mathbb{N}=\left\{1,2,3,...\right\},\;\mathbb{N}_0=\mathbb{N}\cup\left\{0\right\},$$
$\mathbb{R},\;\mathbb{R}_+$ and $\mathbb{C}$ stand for the sets of real, positive real and complex numbers, respectively. The defining series in (\ref{11}) converges in the whole complex $z$-plane when
\begin{equation}
\Delta=\sum_{j=1}^q B_j-\sum_{i=1}^p A_i>-1,
\end{equation}
when $\Delta=0,$ then the series in (\ref{11}) converges for $|z|<\nabla,$ where 
\begin{equation}
\nabla=\left(\prod_{i=1}^p A_i^{-A_i}\right)\left(\prod_{j=1}^q B_j^{B_j}\right).
\end{equation}
If, in the definition (\ref{11}), we set
$$A_1=...=A_p=1\;\;\;\textrm{and}\;\;\;B_1=...=B_q=1,$$
we get the relatively more familiar generalized hypergeometric function ${}_pF_q[.]$ given by
\begin{equation}\label{rrrrr}
{}_p F_q\left[^{\alpha_1,...,\alpha_p}_{\beta_1,...,\beta_q}\Big|z\right]=\frac{\prod_{j=1}^q\Gamma(\beta_j)}{\prod_{i=1}^p\Gamma(\alpha_i)}{}_p\Psi_q\Big[_{(\beta_1,1),...,(\beta_q,1)}^{(\alpha_1,1),...,(\alpha_p,1)}\Big|z \Big]
\end{equation}

The Fox-Wright function appeared recently as a fundamental solutions of diffusion-like equations containing fractional derivatives in time of order less than $1.$ In the physical literature, such equations are in general referred to as fractional sub-diffusion equations, since they are used as
model equations for the kinetic description of anomalous diffusion processes of slow type, characterized by a sub-linear growth of the variance (the mean squared displacement) with time (see for example \cite{MMMM}).

The H-function was introduced by Fox in \cite{fox} as a generalized hypergeometric
function defined by an integral representation in terms of the Mellin-Barnes contour integral
\begin{equation}
\begin{split}H_{q,p}^{m,n}\left(z\Big|^{(B_q,\beta_q)}_{(A_p,\alpha_p)}\right)&=H_{q,p}^{m,n}\left(z\Big|^{(B_1,\beta_1),...,(B_q,\beta_q)}_{(A_1,\alpha_1),...,(A_p,\alpha_p)}\right)\\
&=\frac{1}{2i\pi}\int_{\mathcal{L}}\frac{\prod_{j=1}^m\Gamma (A_j s+\alpha_j)\prod_{j=1}^n\Gamma(1-\beta_j-B_j s)}{\prod_{j=n+1}^q\Gamma (B_k s+\beta_k)\prod_{j=m+1}^p\Gamma(1-\alpha_j-A_j s)}z^{-s}ds.
\end{split}
\end{equation}
Here $\mathcal{L}$ is a suitable contour in $\mathbb{C}$  and $z^{-s}=\exp(-s\log|z|+i\arg(z)),$ where $\log|z|$ represents the natural logarithm of $|z|$ and $\arg(z)$  is not necessarily the principal value. 

The definition of the H-function is still valid when the $A_i$'s and $B_j$'s  are positive rational numbers. Therefore, the H-function contains, as special cases. More importantly, it contains  the Fox-Wright generalized hypergeometric function defined in (\ref{11}), the generalized Mittag-Leffler functions, etc. For example, the function ${}_p\Psi_q[.]$ is one of these special case of H-function. By the definition (\ref{11}) it is easily extended to the complex plane as follows \cite[Eq. 1.31]{AA},
\begin{equation}\label{label}
{}_p\Psi_q\Big[_{(\beta_q,B_q)}^{(\alpha_p,A_p)}\Big|z \Big]=H_{p,q+1}^{1,q}\left(-z\Big|_{(0,1),(B_q,1-\beta_q)}^{(A_p,1-\alpha_p)}\right).
\end{equation}\label{11?}
The special case for which the  H-function reduces to the Meijer G-function is when $A_1=...=A_p=B_1=...=B_q=A,\;A>0.$ In this case,
\begin{equation}\label{,,,}
H_{q,p}^{m,n}\left(z\Big|^{(B_q,\beta_q)}_{(A_p,\alpha_p)}\right)=\frac{1}{A}G_{p,q}^{m,n}\left(z^{1/A}\Big|^{B_q}_{\alpha_p}\right).
\end{equation}
Additionally, when setting $A_i=B_j=1$ in (\ref{11?}) (or $A=1$ in (\ref{,,,})), the H- and Fox-Wright functions  turn readily into the Meijer G-function.\\

Each of the following definitions will be used in our investigation.\\

A real valued function $f,$ defined on an interval $I,$ i called completely monotonic on $I,$ if $f$ has derivatives of all orders and satisfies
\begin{equation}
(-1)^nf^{(n)}(x)\geq0,\;n\in\mathbb{N}_0,\;\;\textrm{and}\;x\in I.
\end{equation}
The celebrated Bernstein Characterization Theorem gives a necessary and sufficient condition that the function $f$ should be completely monotonic for $0<x<\infty$ is that
\begin{equation}
f(x)=\int_0^\infty e^{-xt}d\mu(t),
\end{equation}
where $\mu(t)$ is non-decreasing and the integral converges for $0<x<\infty.$ 

 A function $f$ is said to be absolutely monotonic on  an interval $I,$ if $f$ has derivatives of all orders and satisfies
$$f^{(n)}(x)\geq0,\;x\in I,\;n\in\mathbb{N}_{0}.$$

A positive function $f$ is said to be logarithmically completely monotonic on an interval $I$ if its logarithm $\log f$ satisfies
\begin{equation}
(-1)^n(\log f)^{(n)}(x)\geq0,\;n\in\mathbb{N},\;\;\textrm{and}\;\;x\in I.
\end{equation}
In \cite[Theorem 1.1]{Berg} and \cite[Theorem 4]{Y}, it was found and verified once again that a logarithmically completely monotonic function must be completely monotonic, but not conversely.

An infinitely differentiable function $f:I\longrightarrow[0,\infty)$ is called a Bernstein function on an interval $I,$ if $f^\prime$ is completely monotonic on $I.$ The Bernstein functions on $(0,\infty)$ can be characterized by \cite[Theorem 3.2]{777} which states that a function  $f:(0,\infty)\longrightarrow[0,\infty)$ is a Bernstein function if and only if it admits the representation
\begin{equation}\label{malouk}
f(x)=a+bx+\int_0^\infty (1-e^{-xt})d\mu(t),
\end{equation}
where $a,b\geq0$ and $\mu$ is a measure on $(0,\infty)$ satisfying $\int_0^\infty \min\{1,t\}d\mu(t)<\infty.$ The formula (\ref{malouk}) is called the L\'evy-–Khintchine representation of $f.$ In \cite[pp.161--162, Theorem 3]{888} and \cite[Proposition 5.25]{777},  it was proved that the reciprocal of a Bernstein function is logarithmically completely monotonic.

In \cite[Definition 2.1]{777},  it was defined that a Stieltjes transform is a function  $f:(0,\infty)\longrightarrow[0,\infty)$ which can be written in the following form:
\begin{equation}\label{678}
f(x)=\frac{a}{x}+b+\int_0^\infty \frac{d\mu(t)}{t+x},
\end{equation}
where $a,b$ are  non-negative constant and $\mu$ is a non-negative measure on $(0,\infty)$ such that the integral $\int_0^\infty \frac{d\mu(t)}{t+1}<\infty.$ In \cite[Theorem 2.1]{Berg}  it was proved that a positive Stieltjes transform must be a logarithmically completely monotonic function
on $(0,\infty)$, but not conversely. We define $S$ to be the class of functions representable by (\ref{678}). Functions representable in one of the forms
\begin{equation}\label{t00}
f(z)=a+\int_0^\infty \frac{d\mu(t)}{(z+t)^\alpha}=\frac{b}{z^\alpha}+\int_0^\infty \frac{d\nu(t)}{(1+zt)^\alpha},
\end{equation}
are known as generalized Stieltjes functions of order $\alpha$. Here, $\alpha>0,\;\mu$ and $\nu$ are non-negative  measures supported on $[0,\infty), a,b\geq0$ are constants and we always choose the principal branch of the power function. The measures $\mu$ and $\nu$ are assumed to produce convergent integrals (\ref{t00}) for each $z\in\mathbb{C}\setminus(-\infty,0].$ We denote by $S_\alpha$ to be the class of functions representable by (\ref{t00}).\\

The present sequel to some of the aforementioned investigations is organized as follows. In Section 2, we derive the Laplace integral representations for the Fox H-function $H_{q,p}^{p,0}$ and for  the Fox-Wright function ${}_p\Psi_{q}$. We give a numbers of consequences, some  monotonicity and log-convexity properties for the Fox-Wright function  are researched, and an Tur\'an type inequality  are proved. In Section 3, we find the generalized  Stieltjes transform representation of the Fox-Wright function ${}_{p+1}\Psi_p.$  As applications, we present  some class of  completely monotonic functions related to the Fox-Wright function. In addition, we deduce new Tur\'an type inequalities for this special function. In Section 4, some further applications are proved, firstly, a class of positive definite function related to the Fox H-function are given. As consequences, we find the non-negativity for a class of function involving the Fox H-function. Next, we show that the Fox-Wright function ${}_p\Psi_q[z]$ has no real zeros and all its zeros lie in the open right half plane $\Re(z)>0.$ Moreover, two-sided exponential inequalities for the Fox-Wright function are given, in particular, we gave a generalization of Luke's inequalities. Finally,  by appealing to each of these two-sided exponential inequalities, two sets of two-sided bounding inequalities for generalized Mathieu's type series are proved.

\section{Laplace transform representation and completely monotonic functions for the Fox-Wright functions}

In the first main result we will need a particular case of Fox's H-function defined by
\begin{equation}
H_{q,p}^{p,0}\left(z\Big|^{(B_q,\beta_q)}_{(A_p,\alpha_p)}\right)=\frac{1}{2i\pi}\int_{\mathcal{L}}\frac{\prod_{j=1}^p\Gamma (A_j s+\alpha_j)}{\prod_{k=1}^q\Gamma (B_k s+\beta_k)}z^{-s}ds,
\end{equation}
where $A_j, B_k>0$ and $\alpha_j,\beta_k$ are real.  The contour $\mathcal{L}$ can be either the left loop $\mathcal{L}_-$ starting at $-\infty+i\alpha$ and ending at $-\infty+i\beta$  for some $\alpha<0<\beta$ such that all poles of the integrand lie inside the loop, or the right loop $\mathcal{L}_+$  starting $\infty+i\alpha$ at and ending $\infty+i\beta$ and leaving all poles on the left, or the vertical line $\mathcal{L}_{ic},\;\Re(z)=c,$ traversed upward and leaving all poles of the integrand on the left. Denote the rightmost pole of the integrand by $\gamma:$
$$\gamma=-\min_{1\leq j\leq p}(\alpha_j/A_j).$$ 
Let $$\rho=\left(\prod_{j=1}^p A_j^{A_j}\right)\left(\prod_{k=1}^q B_k^{-B_k}\right),\;\mu=\sum_{j=1}^q\beta_j-\sum_{k=1}^q\alpha_k+\frac{p-q}{2}.$$
Existence conditions of Fox's H-function  under each choice of the contour $\mathcal{L}$ have been thoroughly considered in the book \cite{AA}. Let $z>0$ and under the conditions:
$$\sum_{j=1}^p A_j=\sum_{k=1}^q B_k,\;\;\rho\leq 1,$$
we get that the function $H_{q,p}^{p,0}(z)$ exists by means of \cite[Theorem 1.1]{AA}, if we choose $\mathcal{L}=\mathcal{L}_+$ or $\mathcal{L}=\mathcal{L}_{ic}$ under the additional restriction $\mu> 1.$ Only the second choice of the contour ensures the existence of the Mellin transform of $H_{q,p}^{p,0}(z)$, see \cite[Theorem 2.2]{AA}. In \cite[Theorem 6]{Karp}, the author extend the condition $\mu>1$ to $\mu>0$ and proved that the function $H_{q,p}^{p,0}(z)$ is a compact support.\\

In the course of our investigation, one of the main tools is the following result providing the Laplace transform of the Fox's H-function $z^{-1}H_{q,p}^{p,0}(z).$

\begin{theorem}\label{T1} Suppose that $\mu>0,\;\;\textrm{and}\;\;\sum_{j=1}^pA_j=\sum_{k=1}^qB_k.$
Then, the following integral representation
\begin{equation}\label{fr1}
{}_p\Psi_q\Big[_{(\beta_q, B_q)}^{(\alpha_p, A_p)}\Big|z\Big]=\int_0^\rho e^{zt}H_{q,p}^{p,0}\left(t\Big|^{(B_q,\beta_q)}_{(A_p,\alpha_p)}\right)\frac{dt}{t},\;\;(z\in\mathbb{R}),
\end{equation}
hold true. Moreover,  the function 
$$z\mapsto{}_p\Psi_q\Big[_{(\beta_q, B_q)}^{(\alpha_p, A_p)}\Big|-z\Big]$$
is completely monotonic on $(0,\infty),$ if and only if, the function $H_{q,p}^{p,0}(z)$ is non-negative on $(0,\rho).$
\end{theorem}
\begin{proof} Upon setting $k=s,\;k\in\mathbb{N}_0$ in the Mellin transform for the Fox's H-function $H_{q,p}^{p,0}(z)$ \cite[Theorem 6]{Karp}:
\begin{equation}\label{malek2}
\frac{\prod_{i=1}^p\Gamma(A_i s+\alpha_i)}{\prod_{k=1}^q \Gamma(B_k s+ \beta_k)}=\int_0^\rho H_{q,p}^{p,0}\left(t\Big|^{(B_q,\beta_q)}_{(A_p,\alpha_p)}\right)t^{s-1}dt,\;\Re(s)>\gamma,
\end{equation}
we get 
\begin{equation*}
\begin{split}
{}_p\Psi_q\Big[_{(\beta_q, B_q)}^{(\alpha_p, A_p)}\Big|z\Big]&=\sum_{k=0}^\infty\frac{\prod_{i=1}^p\Gamma(A_i k+\alpha_i)z^k}{k!\prod_{j=1}^q \Gamma(B_j k+ \beta_j)}\\
&=\sum_{k=0}^\infty \int_0^\rho H_{q,p}^{p,0}\left(t\Big|^{(B_q,\beta_q)}_{(A_p,\alpha_p)}\right)\frac{(zt)^{k}}{k!}\frac{dt}{t}\\
&=\int_0^\rho H_{q,p}^{p,0}\left(t\Big|^{(B_q,\beta_q)}_{(A_p,\alpha_p)}\right)\left(\sum_{k=0}^\infty \frac{(zt)^{k}}{k!}\right)\frac{dt}{t}\\
&=\int_0^\rho e^{zt}H_{q,p}^{p,0}\left(t\Big|^{(B_q,\beta_q)}_{(A_p,\alpha_p)}\right)\frac{dt}{t}.
\end{split}
\end{equation*}
For the exchange of the summation and integration, we use the  asymptotic relation \cite[Theorem 1.2, Eq. 1.94]{AA}
\begin{equation}\label{asy}
H_{q,p}^{m,n}(z)=\theta(z^{-\gamma}), \;|z|\longrightarrow0.
\end{equation}
Now, suppose that the function is completely monotonic on $(0,\infty),$ therefore by means of  Bernstein Characterization Theorem and using the fact of  the uniqueness of the measure with given Laplace transform (see \cite[Theorem 6.3]{WI}), we deduce that  $H_{q,p}^{p,0}(z)$ is non-negative on $(0,\rho),$ which evidently completes the proof of Theorem \ref{T1}.
\end{proof}
\begin{corollary}\label{corollary} Suppose that the hypotheses of Theorem \ref{T1} are satisfied. We define the sequence $(\psi_{n,m})_{n,m\geq0}$ by 
$$\psi_{n,m}=\frac{\prod_{i=1}^p\Gamma(\alpha_i+(n+m)A_i)}{\prod_{j=1}^q\Gamma(\beta_j+(n+m)B_j)},\;n,m\in\mathbb{N}_0.$$
If $(H_1^n): \psi_{n,2}<\psi_{n,1}\;\textrm{and}\;\psi_{n,1}^2<\psi_{n,0}\psi_{n,2},\;\textrm{for\;all}\;n\in\mathbb{N}_0,$ then the function $$z\mapsto{}_p\Psi_q\Big[_{(\beta_q, B_q)}^{(\alpha_p, A_p)}\Big|-z\Big]$$
is completely monotonic on $(0,\infty),$ and consequently, the function $H_{q,p}^{p,0}(z)$ is non-negative on $(0,\rho).$
\end{corollary}
\begin{proof} In \cite[Theorem 4]{P}, the authors proved that the function ${}_p\Psi_q\Big[_{(\beta_q,B_q)}^{(\alpha_p,A_p)}\Big|z\Big]$ satisfying the following inequality 
\begin{equation}\label{MMMM}
\psi_{0,0}e^{\psi_{0,1}\psi_{0,0}^{-1}|z|}\leq{}_p\Psi_q\Big[_{(\beta_q, B_q)}^{(\alpha_p, A_p)}\Big|z\Big]\leq  \psi_{0,0}-\psi_{0,1}(1-e^{|z|}),\;z\in\mathbb{R},
\end{equation} 
if $\psi_{0,1}>\psi_{0,2}$ and $\psi_{0,1}^2<\psi_{0,0}\psi_{0,2}.$ On the other hand, by the left hand side of the above inequalities, we get  for $n\geq0$
$$(-1)^n\frac{d^n}{dz^n}{}_p\Psi_q\left[_{(\beta_p, B_q)}^{(\alpha_p, A_p)}\Big|-z\right]={}_p\Psi_q\left[_{(\beta_q+nB_q, B_q)}^{(\alpha_p+nA_p, A_p)}\Big|-z\right]\geq \psi_{n,0}e^{\psi_{n,1}\psi_{n,0}^{-1}|z|}>0.$$
So, the function $z\mapsto{}_p\Psi_q\Big[_{(\beta_q, B_q)}^{(\alpha_p, A_p)}\Big|-z\Big]$
is completely monotonic on $(0,\infty),$ and consequently, the function $H_{q,p}^{p,0}(z)$ is non-negative on $(0,\rho),$ by means of Theorem \ref{T1}.
\end{proof}
\begin{corollary}\label{cc2}Suppose that the hypotheses of Corollary \ref{corollary} are satisfied. In addition, assume that $\lambda,\;\omega>0.$ Then, the function 
$$z\mapsto z^{-\lambda}{}_{p+1}\Psi_q\Big[_{(b_q,B_q)}^{(\lambda,1),(a_p,A_p)}\Big|-\frac{\omega}{z}\Big]$$
is completely monotonic on $(0,\infty).$
\end{corollary}
\begin{proof}From the integral representation \cite[Eq. 7]{P}
$$z^{-\lambda}{}_{p+1}\Psi_q\Big[_{(\beta_q,B_q)}^{(\lambda,1),(\alpha_p, A_p)}\Big|-\frac{\omega}{z}\Big]=\int_0^\infty e^{-zt} t^{\lambda-1}{}_{p}\Psi_q\Big[_{(\beta_q,B_q)}^{(\alpha_p,A_p)}\Big|-\omega t\Big]dt,$$
and using the fact that the function ${}_{p}\Psi_q\Big[_{(\beta_q,B_q)}^{(\alpha_p,A_p)}\Big|-z\Big]dt,$ is non-negative on $(0,\infty),$ we deduce that the function $z^{-\lambda}{}_{p+1}\Psi_q\Big[_{(\beta_q,B_q)}^{(\lambda,1),(\alpha_p, A_p)}\Big|-\frac{\omega}{z}\Big]$ is completely monotonic on $(0,\infty).$
\end{proof}
\noindent\textbf{Remark 1.}
\textbf{a.} Combining (\ref{fr1}) with  (\ref{label}), we obtain
\begin{equation}
H_{p,q+1}^{1,q}\left(z\Big|_{(0,1),(B_q,1-\beta_q)}^{(A_p,1-\alpha_p)}\right)=\int_0^\rho e^{-zt}H_{q,p}^{p,0}\left(t\Big|^{(B_q,\beta_q)}_{(A_p,\alpha_p)}\right)\frac{dt}{t}.
\end{equation}
\textbf{b.} In view of (\ref{fr1}) and (\ref{,,,}), we get 
\begin{equation}\label{!!!!}
{}_p\Psi_p\Big[_{(\beta_q, A)}^{(\alpha_p, A)}\Big|z\Big]=A^{-1}\int_0^1e^{u z}G_{p,p}^{p,0}\left(u^{1/A}\Big|^{\beta_p}_{\alpha_p}\right)\frac{du}{u},\;\;A>0, z\in\mathbb{R}.
\end{equation}
Letting in the above formula, the value $A=1,$ we get \cite[Corollary 1, Eq. 11]{Karp1}
\begin{equation}
{}_pF_p\Big[_{\beta_1,...,\beta_p}^{\alpha_1,...,\alpha_p}\Big|z\Big]=\prod_{j=1}^p\frac{\Gamma(\beta_i)}{\Gamma(\alpha_i)}\int_0^1e^{zt}G_{p,p}^{p,0}\left(t\Big|^{\beta_p}_{\alpha_p}\right)\frac{dt}{t}.
\end{equation}
\begin{theorem}\label{T2}Let $\alpha_i,\;\beta_i,\;i=1,...,p$ be a real number such that 
$$(H_2):\;0<\alpha_1\leq...\leq\alpha_p,\;0<\beta_1\leq...\leq\beta_p,\;\;\sum_{j=1}^k \beta_j-\sum_{j=1}^k \alpha_j\geq0,\;\textrm{for}\;k=1,...,p$$
In addition, assume that $\psi=\sum_{j=1}^p(\beta_j-\alpha_j)>0.$ Then, 
 the function
$$z\mapsto{}_p\Psi_p\Big[_{(\beta_q,A)}^{(\alpha_p,A)}\Big|-z\Big],$$
is completely monotonic on $(0,\infty).$ 
\end{theorem}
\begin{proof}
In \cite[Lemma 2]{Karp23}, the authors proved that the function $G_{p,p}^{p,0}\left(t\Big|^{\beta_p}_{\alpha_p}\right)$ is non-negative on $(0,1)$, and since the hypotheses of thus Theorem implies the hypotheses of Theorem \ref{T1} and so we can  used the integral representation (\ref{!!!!}). Therefore, we deduce that all prerequisites of the Bernstein Characterization Theorem for the complete monotone functions are fulfilled, that is, the function ${}_p\Psi_p\Big[_{(\beta_q,A)}^{(\alpha_p,A)}\Big|-z\Big],$
is completely monotonic on $(0,\infty).$ It is important to mention here that there is another proof for proved  the completely monotonic for the function  ${}_p\Psi_p\Big[_{(\beta_q,A)}^{(\alpha_p,A)}\Big|-z\Big],$ without using the integral representation (\ref{!!!!}). For this we make use the inequalities (\ref{MMMM}). In our case, we have 
$$\psi_{n,m}=\prod_{j=1}^p\frac{\Gamma(\alpha_j+(m+n) A)}{\Gamma(\beta_j+(m+n) A)}.$$
Under the condition $(H_2)$, Alzer \cite[Theorem 10]{Alzer} proved that the function 
$$\varphi:z\mapsto \prod_{j=1}^p\frac{\Gamma(\alpha_j+z)}{\Gamma(\beta_j+z)},$$
is completely monotonic on $(0,\infty)$ this yields that $\varphi(A)\geq\varphi(2A)$
and consequently $\psi_{0,1}>\psi_{0,2}.$ On the other hand,  Bustoz and Ismail \cite{Mourad} proved that the function
$$p(z;a,b)=\frac{\Gamma(z)\Gamma(z+a+b)}{\Gamma(z+a)\Gamma(z+b)},\;a,b\geq0,$$
is completely monotonic on $(0,\infty)$, then the function $z\mapsto p(z;a,b)$ is decreasing on $(0,\infty).$ Now, we choosing $a=b=A,$ we obtain $p(\beta_j; A, A)<p(\alpha_j; A, A)$, thus implies that $\psi_{0,1}^2<\psi_{0,0}\psi_{0,2}.$ So, from the inequality (\ref{MMMM}), we deduce that the function ${}_p\Psi_p\left[_{(\beta_p, A)}^{(\alpha_p, A)}\Big|-z\right]$ is non-negative on $(0,\infty).$ On the other hand, for $n\in\mathbb{N}_0$ we have
$$(-1)^n\frac{d^n}{dz^n}{}_p\Psi_p\left[_{(\beta_p, A)}^{(\alpha_p, A)}\Big|-z\right]={}_p\Psi_p\left[_{(\beta+nA, A)}^{(\alpha_p+nA, A)}\Big|-z\right]={}_p\Psi_p\left[_{(\lambda_p, A)}^{(\delta_p, A)}\Big|-z\right]\geq0,$$
where $\delta_p=\alpha_p+nA$ and $\lambda_p=\beta_p+nA$ satisfying the hypothesis $(H_2).$ Thus implies that the function  ${}_p\Psi_p\left[_{(\beta+nA,A)}^{(\alpha_p+nA, A)}\Big|-z\right]$ is non-negative on $(0,\infty),$ and consequently the function ${}_p\Psi_p\Big[_{(\beta_p, A)}^{(\alpha_p, A)}\Big|-z\Big]$ is completely monotonic on $(0,\infty).$ This completes the proof of Theorem \ref{T2}.
\end{proof}
\noindent\textbf{Remark 2.} We see in the second proof of the above Theorem that the condition $\psi>0$ is not necessary for proved the complete monotonicity property for the function ${}_p\Psi_p\Big[_{(\beta_p, A)}^{(\alpha_p, A)}\Big|-z\Big],$ and consequently this property is also true under the conditions hypotheses $(H_2)$ only, and consequently the H-function $H_{p,p}^{p,0}\big[t|^{(A,\beta_p)}_{(A,\alpha_p)}\big]$ is non-negative under thus hypotheses.

\begin{corollary} Keeping the notation and constraints of Theorem \ref{T2}. Then, the function
$$A\mapsto{}_p\Psi_p\Big[_{(\beta_q,A)}^{(\alpha_p,A)}\Big|z\Big],$$
is log-convex on $(0,\infty)$  for all $z\in\mathbb{R}.$ Furthermore, then the following Tur\'an type inequality 
\begin{equation}\label{turan}
{}_p\Psi_p\Big[_{(\beta_q,A)}^{(\alpha_p,A)}\Big|z\Big]{}_p\Psi_p\Big[_{(\beta_q,A+2)}^{(\alpha_p,A+2)}\Big|z\Big]-\left({}_p\Psi_p\Big[_{(\beta_q,A+1)}^{(\alpha_p,A+1)}\Big|z\Big]\right)^2\geq0,
\end{equation}
holds true for all $A\in(0,\infty)$ and $z\in\mathbb{R}.$
\end{corollary}
\begin{proof}
Rewriting the integral representation (\ref{!!!!}) in the following form:
\begin{equation}\label{kol}
{}_p\Psi_p\Big[_{(\beta_q,A)}^{(\alpha_p,A)}\Big|z\Big]=\int_0^1e^{t^A z}G_{p,p}^{p,0}\left(t\Big|^{\beta_p}_{\alpha_p}\right)\frac{dt}{t},\;\;A>0,\;z\in\mathbb{R}.
\end{equation}
Let us recall the H\"older inequality \cite[p. 54]{MIT}, that is
\begin{equation}
\int_a^b|f(t)g(t)|dt\leq\left[\int_a^b|f(t)|^pdt\right]^{1/p}\left[\int_a^b|g(t)|^pdt\right]^{1/q},
\end{equation}
where $p\geq1,\;\frac{1}{p}+\frac{1}{q}=1,\;f$ and $g$ are real functions defined on $(a,b)$ and $|f|^p,\;|g|^q$ are integrable functions on $(a,b)$. From the H\"older's inequality and integral representation (\ref{kol}) and using the fact that the function $A\mapsto x^A$ is convex on $(0,\infty)$ when $x>0.$ For $A_1,A_2>0$ and $t\in[0,1],$ we thus get
\begin{equation}
\begin{split}
 {}_{p}\Psi_p\left[_{(\beta_p,\;tA_1+(1-t)A_2)}^{(\alpha_p,\;tA_1+(1-t)A_2}\Big|z\right]&=\int_0^1 e^{zu^{tA_1+(1-t)A_2}} G_{p,p}^{p,0}\left(u\Big|^{\beta_p}_{\alpha_p}\right)\frac{du}{u}\\
&\leq\int_0^1 e^{tzu^{A_1}} e^{(1-t)zu^{A_2}} G_{p,p}^{p,0}\left(u\Big|^{\beta_p}_{\alpha_p}\right)\frac{du}{u}\\
&=\int_0^1 \left[\frac{e^{zu^{A_1}}}{u} G_{p,p}^{p,0}\left(u\Big|^{\beta_p}_{\alpha_p}\right) \right]^t \left[\frac{e^{zu^{A_2}}}{u} G_{p,p}^{p,0}\left(u\Big|^{\beta_p}_{\alpha_p}\right)\right]^{1-t}du\\
&\leq \left[\int_0^1 e^{zu^{A_1}}  G_{p,p}^{p,0}\left(u\Big|^{\beta_p}_{\alpha_p}\right)\frac{du}{u}\right]^t\left[\int_0^1 e^{zu^{A_2}}  G_{p,p}^{p,0}\left(u\Big|^{\beta_p}_{\alpha_p}\right)\frac{du}{u}\right]^{1-t}\\
&=\left[{}_{p}\Psi_p\left[_{(\beta_p,\;A_1)}^{(\alpha_p,\;A_1)}\Big|z\right]\right]^t\left[{}_{p}\Psi_p\left[_{(\beta_p,\;A_2)}^{(\alpha_p,\;A_2)}\Big|z\right]\right]^{1-t},
\end{split}
\end{equation}
and hence the required result follows. Now, choosing  $A_1=A, A_2=A+2$ and $t=\frac{1}{2}$ in the above inequality we get the Tur\'an type inequality (\ref{turan}).
\end{proof}
\noindent \textbf{Remark 3.} Suppose the hypotheses of Theorem \ref{T2} are satisfied. Repeating the same calculations in Corollary \ref{cc2}, we deduce that the function 
$$z\mapsto z^{-\lambda}{}_{p+1}\Psi_p\Big[_{(\beta_q,A)}^{(\lambda,1),(\alpha_p,A)}\Big|-\frac{1}{z}\Big]$$
is completely monotonic on $(0,\infty),$ and consequently, the Hypergeometric function 
$$z\mapsto z^{-\lambda}{}_{p+1}F_p\Big[_{\;\beta_1,...,\beta_p}^{\;\lambda,\alpha_1,...,\alpha_p}\Big|-\frac{1}{z}\Big]$$
is completely monotonic on $(0,\infty).$ $($see \cite[Theorem 3]{Karp1}.$)$  

\begin{theorem}\label{TTTT3}The function $z\mapsto{}_p\Psi_q\left[_{(\beta_q, B_q)}^{(\alpha_p, A_p)}\Big|\frac{1}{z}\right]$ admits the following Laplace integral representation
\begin{equation}\label{khaledmehrez}
{}_p\Psi_q\Big[_{(\beta_q, B_q)}^{(\alpha_p, A_p)}\Big|\frac{1}{z}\Big]=\int_0^\infty e^{-zt}\left({}_p\Psi_{q+1}\Big[_{(\beta_q+1, B_q),(2,1)}^{\;\;\;(\alpha_p+1,A_p)}\Big|t\Big]+\frac{\prod_{i=1}^p\Gamma(\alpha_i)}{\prod_{j=1}^q\Gamma(\beta_j)}\delta_0\right)dt,
\end{equation}
where $\delta_0$ is the Dirac measure with mass $1$ concentrated at zero. Moreover, the function 
$$z\mapsto{}_p\Psi_q\left[_{(\beta_q, B_q)}^{(\alpha_p, A_p)}\Big|\frac{1}{z}\right]$$ 
is completely monotonic on $(0,\infty).$
\end{theorem}
\begin{proof}Straightforward calculation would yield
$$\int_0^\infty e^{-zt}\left({}_p\Psi_{q+1}\Big[_{(\beta_q+1,B_q),(2,1)}^{\;\;\;(\alpha_p+1, A_p)}\Big|t\Big]+\frac{\prod_{i=1}^p\Gamma(\alpha_i)}{\prod_{j=1}^q\Gamma(\beta_j)}\delta_0\right)dt=$$
\begin{equation*}
\begin{split}
\;\;\;\;\;\;\;\;\;\;\;\;\;\;\;\;\;\;\;\;\;\;\;\;\;\;\;\;\;\;\;\;\;\;\;&=\sum_{m=0}^\infty\frac{\prod_{i=1}^p\Gamma(\alpha_i+A_i m+1)}{\prod_{j=1}^q\Gamma(\beta_i+B_i m+1)\Gamma(m+2)m!}\int_0^\infty t^m e^{-zt} dt+\frac{\prod_{i=1}^p\Gamma(\alpha_i)}{\prod_{j=1}^q\Gamma(\beta_j)}\\
&=\sum_{m=0}^\infty\frac{\prod_{i=1}^p\Gamma(\alpha_i+A_i m+1)}{\prod_{j=1}^q\Gamma(\beta_j+B_j m+1)(m+1)!z^{m+1}}+\frac{\prod_{i=1}^p\Gamma(\alpha_i)}{\prod_{j=1}^q\Gamma(\beta_j)}\\
&=\sum_{m=1}^\infty\frac{\prod_{i=1}^p\Gamma(\alpha_i+A_i m)}{\prod_{j=1}^q\Gamma(\beta_j+B_j m)m!z^{m}}+\frac{\prod_{i=1}^p\Gamma(\alpha_i)}{\prod_{j=1}^q\Gamma(\beta_j)}\\
&=\sum_{m=0}^\infty\frac{\prod_{i=1}^p\Gamma(\alpha_i+A_i m)}{\prod_{j=1}^q\Gamma(\beta_j+B_j m)m!z^{m}}\\
&={}_p\Psi_q\Big[_{(\beta_q,B_q)}^{(\alpha_p,A_p)}\Big|\frac{1}{z}\Big].
\end{split}
\end{equation*}
Therefore, the integral representation (\ref{khaledmehrez}) of the function Fox-Wright function ${}_p\Psi_q\Big[_{(\beta_q,B_q)}^{(\alpha_p,A_p)}\Big|\frac{1}{z}\Big]$ is fulfilled.  Simultaneously, the function ${}_p\Psi_{q+1}\Big[_{(\beta_q+1,B_q),(2,1)}^{(\alpha_p+1,A_p)}\Big|t\Big]$  being positive, all prerequisites of the Bernstein Characterization Theorem for the complete monotone functions are fulfilled, that is, the function ${}_p\Psi_q\Big[_{(\beta_q,B_q)}^{(\alpha_p,A_p)}\Big|\frac{1}{z}\Big]$ is completely monotone on $(0,\infty).$  It is important to mention here that there is another proof for the completely monotone of the Fox-Wright function ${}_p\Psi_q\Big[_{(\beta_q,B_q)}^{(\alpha_p,A_p)}\Big|\frac{1}{z}\Big]$. By using the fact that if the function $f(x)$ is absolutely monotonic then  the function $f(1/x)$ is completely monotonic \cite[p. 151]{WI}, and since the function ${}_p\Psi_q\Big[_{(\beta_q,B_q)}^{(\alpha_p,A_p)}\Big|z\Big]$ is absolutely monotonic and consequently the function ${}_p\Psi_q\Big[_{(\beta_q,B_q)}^{(\alpha_p,A_p)}\Big|\frac{1}{z}\Big]$ is completely monotonic on $(0,\infty),$ which evidently completes the proof of Theorem \ref{TTTT3}.
\end{proof}
\section{Stieltjes transform representation for the Fox-Wright functions and its consequences}

In this section, we show that the Fox-Wright function 
$${}_{p+1}\Psi_q\left[^{(\sigma,1),(\alpha_p,A_p)}_{(\beta_p,B_p)}\Big|-z\right]$$ 
is a generalized  Stieltjes functions of order $\sigma.$
As applications, some class of  logarithmically completely monotonic functions related to the Fox-Wright function are derived. Moreover, we deduce new Tur\'an type inequalities for thus special function.

\begin{theorem} \label{T7}Let $\sigma>0$ and $z\in\mathbb{C}$ such that $|\arg z|<\pi\;\textrm{and}\;|z|<1.$  Assume that the hypotheses of Corollary \ref{corollary} are satisfied. Then, the following Stieltjes transform hold true: 
\begin{equation}\label{malek}
{}_{p+1}\Psi_q\left[^{(\sigma,1),(\alpha_p,A_p)}_{(\beta_q,B_q)}\Big|-z\right]=\int_0^\rho \frac{d\mu(t)}{(1+tz)^\sigma},
\end{equation}
where
\begin{equation}\label{malek1}
d\mu(t)=H_{q,p}^{p,0}\left(t\Big|^{(B_q,\beta_q)}_{(A_p,\alpha_p)}\right)\frac{dt}{t}.
\end{equation}
Furthermore, the function 
$$z\mapsto  {}_{p+1}\Psi_q\left[^{(\sigma,1),(\alpha_p,A_p)}_{(\beta_q,B_q)}\Big|-z\right]$$
is completely monotonic on $(0,1).$
\end{theorem}
\begin{proof}Consider the right-hand side of (\ref{malek}) with $d\mu(t)$ is given by (\ref{malek1}). We make use of the formula (\ref{malek2}) and applying the  binomial expansion to 
$$(1+z)^{-\sigma}=\sum_{k=0}^\infty(\sigma)_k\frac{(-1)^kz^k}{k!},\;\;z\in\mathbb{C}\;\; \textrm{such that}\;\;|z|<1,$$ 
and integrating term by term we obtain the left-hand side of (\ref{malek}). Finally, its easy to see that the function $z\mapsto{}_{p+1}\Psi_q\left[^{(\sigma,1),(\alpha_p,A_p)}_{(\beta_q,B_q)}\Big|-z\right]$ is completely monotonic on $(0,1).$ 
\end{proof}
\begin{corollary} Let $0<\sigma\leq 1$. Assume that the hypotheses of Corollary \ref{corollary} are satisfied. Then the following assertions are true:\\
\textbf{a.}The function 
$$z\mapsto{}_{p+1}\Psi_q\left[^{(\sigma,1),(\alpha_p,A_p)}_{(\beta_q,B_q)}\Big|-z\right]$$
is logarithmically completely monotonic on $(0,1).$\\
\textbf{b.}The function
\begin{equation}
z\mapsto 1\Big/\;{}_{p+1}\Psi_q\left[^{(\sigma,1),(\alpha_p,A_p)}_{(\beta_q,B_q)}\Big|-z\right]
\end{equation}
is a  Bernstein function  on $(0,1).$ In particular, the function 
$$z\mapsto  {}_{p+1}\Psi_q\left[^{(\sigma+1,1),(\alpha_p+A_p,A_p)}_{(\beta_q+B_q,B_q)}\Big|-z\right]\Big/{}_{p+1}\Psi_q\left[^{(\sigma,1),(\alpha_p,A_p)}_{(\beta_q,B_q)}\Big|-z\right]$$
is completely monotonic on $(0,1).$
\end{corollary}
\begin{proof}a. By using the fact that $S_\alpha\subseteq S_\beta$ whenever $\alpha\leq\beta$ (see \cite{SO}), we deduce
$${}_{p+1}\Psi_q\left[^{(\sigma,1),(\alpha_p,A_p)}_{(\beta_q,B_q)}\Big|-z\right]\in S_1=S,$$
where $0<\sigma\leq1.$ On the other hand,  it was proved in \cite[Theorem 1.2]{Berg} that  the set  of Stieltjes transforms $\mathcal{S}\setminus\{0\}$  is a proper subset of the class of logarithmically completely monotonic functions.\\
b. The result follows from Theorem \ref{T7} and Proposition 1.3 \cite{Berg1}.
\end{proof}
\begin{theorem}\label{TT5} Under the assumptions stated in Corollary \ref{corollary}. The function $$\sigma\mapsto\Xi(\sigma)={}_{p+1}\Psi_q\left[^{(\sigma,1),(\alpha_p,A_p)}_{(\beta_q,B_q)}\Big|z\right]$$
is log-convex on $(0,\infty)$ for each $z\in(0,1).$ Furthermore, the following Tur\'an type inequality
\begin{equation}\label{turan1}
{}_{p+1}\Psi_q\left[^{(\sigma,1),(\alpha_p,A_p)}_{(\beta_q,B_q)}\Big|z\right]{}_{p+1}\Psi_q\left[^{(\sigma+2,1),(\alpha_p,A_p)}_{(\beta_q,B_q)}\Big|z\right]-\left({}_{p+1}\Psi_q\left[^{(\sigma+1,1),(\alpha_p,A_p)}_{(\beta_q,B_q)}\Big|z\right]\right)^2\geq0,
\end{equation}
holds true for all $\sigma\in(0,\infty)$ and $z\in(0,1).$
\end{theorem}
\begin{proof} Recall the Chebyshev integral inequality \cite[p. 40]{DM}: if $f,g:[a,b]\longrightarrow\mathbb{R}$ are synchoronous (both increasing  or decreasing) integrable functions, and $p:[a,b]\longrightarrow\mathbb{R}$  is a positive integrable function, then 
\begin{equation}\label{OO}
\int_a^b p(t)f(t)dt\int_a^b p(t)g(t)dt\leq \int_a^b p(t)dt\int_a^b p(t)f(t)g(t)dt.
\end{equation}
Note that if $f$ and $g$ are asynchronous (one is decreasing and the other is increasing),
then (\ref{OO}) is reversed. Let $\sigma_2>\sigma_1\geq0$ and arbitrary $\epsilon>0$ and we consider the functions $p,f,g:[0,\rho]\longrightarrow\mathbb{R}$ defined by:
$$p(t)=\frac{t^{-1}H_{q,p}^{p,0}\left(t\Big|^{(B_q,\beta_q)}_{(A_p,\alpha_p)}\right)}{(1-zt)^{\sigma_1}},\;\;f(t)=\frac{1}{(1-zt)^{\sigma_2-\sigma_1}},\;\;g(t)=\frac{1}{(1-zt)^\epsilon}.$$
Since the function $p$ is non-negative on $(0,\rho)$ and the functions $f$ and $g$ are increasing on 
$(0,\rho)$ if $z\in(0,1),$ we gave
\begin{equation*}
\Xi(\sigma_1+\epsilon)\Xi(\sigma_2)\leq\Xi(\sigma_1)\Xi(\sigma_2+\epsilon).
\end{equation*}
The above inequality is equivalent to log-convexity for the function $\sigma\mapsto\Xi(\sigma)$ on $(0,\infty)$ for each $z\in(0,1)$ (see \cite[Chapter I.4]{MM}). Now, focus on the Tur\'an type inequality (\ref{turan1}). Since the function  $\sigma\mapsto\Xi(\sigma)$ is log-convex on $(0,\infty)$ for each $x\in(0,1).$  it follows that for all $\sigma_1,\sigma_2>0,\;t\in[0,1]$ and $x\in(0,1),$ we have
$$\Xi(t\sigma_1+(1-t)\sigma_2)\leq[\Xi(\sigma_1)]^t[\Xi(\sigma_2)]^{1-t}.$$
Upon setting $$\sigma_1=\sigma,\;\sigma_2=\sigma+2\;\;\textrm{and}\;\;t=\frac{1}{2},$$  the above inequality reduces to the Tur\'an type inequality (\ref{turan1}), which evidently completes the proof of Theorem \ref{TT5}.
\end{proof}
\noindent\textbf{Remark 4.} If we consider the assumptions $(H_2)$ in Theorem \ref{T2}, we get again the results proved in Section 3, where $q=p$ and $A=A_p=B_p.$ 
\section{Further Applications}
\subsection{ A class of positive definite functions related to the Fox H-function} The purpose of this  section is to prove a class of positive definite functions related to the Fox H-function. As an application, we derive a class of function involving the Fox H-function is non-negative. Let us remind the reader that a continuous function $f:\mathbb{R}^d\longrightarrow\mathbb{C}$ is called positive definite function, if for all $N\in\mathbb{N},$ all sets of pairwise distinct centers $X=\left\{x_1,...,x_N\right\}\subseteq\mathbb{R}^d$ and $z=\left\{\xi_1,...,\xi_N\right\}\subset \mathbb{C}^N,$ the quadratic form
$$\sum_{j=1}^N\sum_{k=1}^N\xi_j\bar{\xi_k}f(x_j-x_k)$$
is non-negative.
\begin{theorem}\label{malouk}Let the parameters $\rho,\nu\in\mathbb{C},$ satisfy the conditions
$$\Re(\rho)+\Re(\nu)+\min_{1\leq j\leq p}\left[\frac{\alpha_j}{A}\right]>-1,$$
$$\Re(\nu)>-\frac{1}{2}\;\;\textrm{and}\;\;\Re(\rho)+\Re(\nu)<\frac{3}{2}.$$
Moreover, assume that the hypotheses $(H_2)$ of Theorem \ref{T2} are satisfies. Then the function 
\begin{equation}
\chi:z\mapsto z^{-(\rho+\nu)}H_{p,p+2}^{1,p}\left[2z\Big|_{(\frac{1}{2},\frac{\rho+\nu}{2}),(A,1-\beta_p),(\frac{1}{2},\frac{\rho-\nu}{2})}^{\;\;\;\;\;\;\;\;\;\;\;\;\;\;\;(A,1-\alpha_p)}\right]
\end{equation}
is positive definite function on $\mathbb{R}.$
\end{theorem}
\begin{proof}We can write the following formula \cite[Eq. (2.45), pp. 57]{AA}
\begin{equation}
\int_0^\infty x^{\rho-1} J_\nu(z x) H_{q,p}^{m,n}\left[x\Big|^{(B_q,\beta_q)}_{(A_p,\alpha_p)}\right]dx=\frac{2^{\rho-1}}{z^\rho}H_{q+2,p}^{m,n+1}\left[\frac{2}{z}\Big|^{(\frac{1}{2},1-\frac{\rho+\nu}{2}),(B_q,\beta_q),(\frac{1}{2},1-\frac{\rho-\nu}{2})}_{\;\;\;\;\;\;\;\;\;\;(A_p,\alpha_p)}\right]
\end{equation}
in the following form
\begin{equation}\label{sou}
\int_0^\infty x^{\rho+\nu-1} \mathcal{J}_\nu(z x) H_{q,p}^{m,n}\left[x\Big|^{(B_q,\beta_q)}_{(A_p,\alpha_p)}\right]dx=\frac{\Gamma(\nu+1)2^{\rho+\nu-1}}{z^{\rho+\nu}}H_{q+2,p}^{m,n+1}\left[\frac{2}{z}\Big|^{(\frac{1}{2},1-\frac{\rho+\nu}{2}),(B_q,\beta_q),(\frac{1}{2},1-\frac{\rho-\nu}{2})}_{\;\;\;\;\;\;\;\;\;\;(A_p,\alpha_p)}\right]
\end{equation}
where 
$$\mathcal{J}_\nu(x)=2^\nu\Gamma(\nu+1)\frac{J_\nu(x)}{x^\nu},\;\Re(\nu)>-\frac{1}{2},$$
with $J_\nu(x)$ is the Bessel function of index $\nu.$ On the other hand, as the function $\mathcal{J}_\nu(x)$ is positive definite function \cite[Proposition 2]{KKKK} and the function $H_{p,p}^{p,0}[t|_{(A,\alpha_p)}^{(A,\beta_p)}]$ is non-negative (Remark 2), we deduce that for any finite list of complex numbers $\xi_1,...,\xi_N$ and $z_1,...,z_N\in\mathbb{R},$ 
\begin{equation}
\sum_{j=1}^N\sum_{k=1}^N\xi_j\bar{\xi_k}(z_j-z_k)^{-(\rho+\nu)}H_{p+2,p}^{p,1}\left[2(z_j-z_k)^{-1}\Big|^{(\frac{1}{2},1-\frac{\rho+\nu}{2}),(A,\beta_q),(\frac{1}{2},1-\frac{\rho-\nu}{2})}_{\;\;\;\;\;\;\;\;\;\;\;\;\;(A,\alpha_p)}\right]=
\end{equation}
$$=\frac{1}{\Gamma(\nu+1)2^{\rho+\nu-1}}\int_0^\infty  x^{\rho+\nu-1}\left[\sum_{j=1}^N\sum_{k=1}^N\xi_j\bar{\xi_k}\mathcal{J}_\nu(xz_j-xz_k)\right] H_{p,p}^{p,0}\left[x\Big|^{(A,\beta_p)}_{(A,\alpha_p)}\right]dx\geq0.$$
Thus, implies that the function 
$$\chi_1:z\mapsto z^{-(\rho+\nu)}H_{p+2,p}^{p,1}\left[\frac{2}{z}\Big|^{(\frac{1}{2},1-\frac{\rho+\nu}{2}),(A,\beta_p),(\frac{1}{2},1-\frac{\rho-\nu}{2})}_{\;\;\;\;\;\;\;\;\;\;\;\;\;\;\;(A,\alpha_p)}\right]$$
is positive definite function on $\mathbb{R}.$ So, the \cite[Property 1.3, p. 11]{AA} completes the proof of Theorem \ref{malouk}.
\end{proof}
\begin{theorem}Let the parameters $\rho,\nu\in\mathbb{C},$ satisfy the conditions
\begin{equation}\label{hhh}
\Re(\rho)+\Re(\nu)+\min_{1\leq j\leq p}\left[\frac{\alpha_j}{A}\right]>0,\;\Re(\nu)>-\frac{1}{2},
\end{equation}
\begin{equation}
\Re(1-(\rho+\nu))+\max_{1\leq j\leq p}\left(\frac{\alpha_j}{A}\right)<1,\;\textrm{and}\;\;\Re(\rho)+\Re(\nu)<\frac{3}{2}.
\end{equation}
Then, the function
\begin{equation}
\mathcal{K}_{p,q}^{\nu,\rho}(z)=z^{\rho+\nu-1}H_{p+2,p+2}^{p+1,1}\left[8z\Big|_{(\frac{1}{2},\frac{1-(\rho+\nu)}{2}),(A, \alpha_p),(\frac{1}{2},1-\frac{\rho+\nu}{2})}^{(\frac{1}{2},1-\frac{\rho+\nu}{2}),(A, \beta_p),(\frac{1}{2},1-\frac{\rho-\nu}{2})}\right]
\end{equation}
is non-negative on $\mathbb{R}.$
\end{theorem}
\begin{proof}Firstly, we proved that the function $z\longmapsto\chi(z)$ is in $L^1(0,\infty).$ By using the asymptotic expansion \cite[Eq. 1.94, pp. 19]{AA}
$$H_{p,q}^{m,n}\left[z\Big|^{(A_p, a_p)}_{(B_q, b_q)}\right]=\theta(z^c),\;|z|\longrightarrow0,\;\textrm{
where}\;c=\min_{1\leq j\leq m}\left[\frac{\Re(b_j)}{B_j}\right].$$
In our case  $m=1,b_1=\frac{\rho+\nu}{2}$ and $B_1=\frac{1}{2},$ and consequently $c=\rho+\nu.$ Thus implies that 
\begin{equation}\label{sawsen}
\chi(z)=\theta(1)\;\;\textrm{as}\;\;z\longrightarrow0.
\end{equation}
On the other hand, by using the asymptotic \cite[Eq. 1.94, pp. 19]{AA}
$$H_{p,q}^{m,n}\left[z\Big|^{(A_p, a_p)}_{(B_q, b_q)}\right]=\theta(z^d),\;|z|\longrightarrow\infty,\;\textrm{
where}\;d=\min_{1\leq j\leq n}\left[\frac{\Re(a_j)-1}{A_j}\right].$$
In our case $n=p$ and $a_j=1-\alpha_j$ and consequently $d=-(\Re(\rho)+\Re(\nu)+\min_{1\leq j\leq p}\frac{\Re(\alpha_j)}{A}),$ thus we get
\begin{equation}\label{sana}
\chi(z)=\theta\left(z^{-(\rho+\nu+M)}\right),\;\textrm{where}\; M=\min_{1\leq j\leq p}\frac{\alpha_j}{A}.
\end{equation}
Now, combining (\ref{sana}) with the hypotheses (\ref{hhh}) and (\ref{sawsen}), we deduce that the function $z\longmapsto\chi(z)$ is in $L^1(0,\infty).$ In addition, as  $z\mapsto\mathcal{J}_\nu(zx)$ is an even function and using the integral representation (\ref{sou}), we deduce that  $z\longmapsto\chi(z)$ is an even function, and consequently thus function is in $L^1(\mathbb{R}).$\\
Secondly, we calculate the Fourier transform of the function $z\longmapsto\chi(z).$ Since $z\longmapsto\chi(z)$ is an even function then the the Fourier transform  can be written as a Hankel transform (see \cite[Lemma 1.1]{AK}, when $\alpha=-1/2)$, more precisely, 
\begin{equation}
\mathcal{F}\left(\chi\right)(z)=\sqrt{\frac{2}{\pi}}\int_0^\infty \chi(x)\cos(xz)dz.
\end{equation}
We now make use of the following formula \cite[Eq. 2.50, p. 58]{AA}
\begin{equation}
\int_0^\infty x^{\rho-1}\cos(xz) H_{p,q}^{m,n}\left[x\Big|^{(A_p, a_p)}_{(B_q, b_q)}\right]dx=\frac{2^{\rho-1}\sqrt{\pi}}{z^\rho}H_{p+2,q}^{m,n+1}\left[\frac{2}{z}\Big|^{(\frac{1}{2},\frac{2-\rho}{2}),(A_p, a_p),(\frac{1}{2},\frac{1-\rho}{2})}_{\;\;\;\;\;\;\;\;\;\;(B_q, b_q)}\right]
\end{equation}
where $z>0,\;\rho\in\mathbb{C}$ such that
$$\Re(\rho)+\min_{1\leq j\leq m}\Re\left(\frac{b_j}{B_j}\right)>0\;\;\textrm{and}\;\;\Re(\rho)+\max_{1\leq j\leq n}\Re\left(\frac{a_j-1}{A_j}\right)<1.$$
In our case $\rho\longrightarrow 1-(\rho+\nu),\;m=1\;n=p,\;b_1=\frac{\rho+\nu}{2}$ and $B_1=\frac{1}{2},$ thus $$\Re(\rho)+\min_{1\leq j\leq m}\left(\frac{b_j}{B_j}\right)=1>0.$$
Therefore,
\begin{equation}
\begin{split}
\mathcal{F}\left(\chi\right)(z)&=2^{1-\rho-\nu}z^{\rho+\nu-1}H_{p+2,p+2}^{1,p+1}\left[\frac{8}{z}\Big|^{(\frac{1}{2},\frac{1+\rho+\nu}{2}),(A, 1-\alpha_p),(\frac{1}{2},\frac{\rho+\nu}{2})}_{(\frac{1}{2},\frac{\rho+\nu}{2}),(A, 1-\beta_p),(\frac{1}{2},\frac{\rho-\nu}{2})}\right]\\
&=2^{1-\rho-\nu}z^{\rho+\nu-1}H_{p+2,p+2}^{p+1,1}\left[8z\Big|_{(\frac{1}{2},\frac{1-(\rho+\nu)}{2}),(A, \alpha_p),(\frac{1}{2},1-\frac{\rho+\nu}{2})}^{(\frac{1}{2},1-\frac{\rho+\nu}{2}),(A, \beta_p),(\frac{1}{2},1-\frac{\rho-\nu}{2})}\right].
\end{split}
\end{equation}
Finally, using the fact that the Fourier transform for a function in $L^1$ and positive definite function  is non-negative ( see for example \cite[theorem 6.6]{FD} or \cite[Theorem 6.11, p. 74]{HW}). So, the proof is completes.
\end{proof}
\subsection{Zeros of the Fox-Wright functions}
\begin{theorem} Keeping the notation and constraints of Corollary \ref{corollary}. Then, all the roots of the Fox-Wright function ${}_p\Psi_q\left[^{(\beta_p,B_q)}_{(\alpha_p,A_p)}\Big|z\right]$ are in the left-hand half-plane $\Re z\leq 0.$
\end{theorem}
\begin{proof}By using the following identity \cite[Theorem 8]{Karp}
\begin{equation}
H_{q,p}^{p,0}\left(z\Big|^{(\beta_q,B_q)}_{(\alpha_p,A_p)}\right)=\frac{1}{\log(\rho/z)}\int_{z/\rho}^1H_{q,p}^{p,0}\left(\frac{z}{u}\Big|^{(\beta_q,B_q)}_{(\alpha_p,A_p)}\right)\frac{Q(u)}{u}du
\end{equation}
where $Q(u)$ is defined by
$$Q(u)=\sum_{i=1}^p\frac{t^{\alpha_i/A_i}}{1-t^{1/A_i}}-\sum_{j=1}^q\frac{t^{\beta_j/B_j}}{1-t^{1/B_j}},\;t\in(0,1),$$
we deduce that the function 
$$t\mapsto H_{q,p}^{p,0}\left(t\Big|^{(\beta_p,B_q)}_{(\alpha_p,A_p)}\right)$$
 decreasing on $(0, 1).$ On the other hand, by means of the integral representation (\ref{fr1}) and we make the following change of variables $t=1-u$ we get
\begin{equation}\label{yui}
e^{-z}{}_p\Psi_q\Big[_{(\beta_q,B_q)}^{(\alpha_p,A_p)}\Big|z\Big]=\int_{1-\rho}^1 e^{-zu}H_{q,p}^{p,0}\left(1-u\Big|^{(B_q,\beta_p)}_{(A_p,\alpha_p)}\right)\frac{du}{1-u}
\end{equation}
Taking into account the obvious equation \cite[Property 2.5, Eq. 2.1. 5]{AA}
\begin{equation}
z^\sigma H_{q,p}^{m,n}\left(z\Big|_{(B_q,\beta_q)}^{(A_p,\alpha_p)}\right)= H_{q,p}^{m,n}\left(z\Big|_{(B_q,\beta_q+\sigma B_q)}^{(A_p,\alpha_p+\sigma A_p)}\right),\;\sigma\in\mathbb{C},
\end{equation}
and using (\ref{yui}) we get 
\begin{equation}
e^{-z}{}_p\Psi_q\Big[_{(\beta_q,B_q)}^{(\alpha_p,A_p)}\Big|z\Big]=\int_{1-\rho}^1 e^{-zu}H_{q,p}^{p,0}\left(1-u\Big|^{(B_q,\beta_q-B_q)}_{(A_p,\alpha_p-A_p)}\right)du.
\end{equation}
Since the function $H_{q,p}^{p,0}\left(1-u\Big|^{(B_q,\beta_q-B_q)}_{(A_p,\alpha_p-A_p)}\right)$ is non-negative and increasing on $(0, \rho),$ we deduce that the hypothesis of Theorem 2.1.7 in \cite{SE} is fulfilled.
\end{proof}
\subsection{Extended Luke's inequalities}
Our aim in the this section is to present  two-sided exponential inequalities for the Fox-Wright function. As an application, we gave a generalization of Luke's inequalities.

\begin{theorem}\label{T6} Suppose that the hypotheses of Corollary \ref{corollary} be satisfied, then the following inequalities holds:
\begin{equation}\label{§§}
\psi_{0,0} e^{-\psi_{0,1}\psi_{0,0}^{-1}z}\leq{}_{p}\Psi_q\Big[_{(\beta_q, B_q)}^{(\alpha_p, A_p)}\Big|-z\Big]\leq\psi_{0,0}-\frac{\psi_{0,1}}{\rho}(1-e^{-\rho z}),\;z\in\mathbb{R}.
\end{equation}
\end{theorem}
\begin{proof} We recall the Jensen's integral inequality \cite[ Chap. I, Eq. (7.15)]{MI},
\begin{equation}\label{C}
\varphi\left(\int_a^b f(s)d\mu(s)\Big/\int_a^b d\mu(s)\right)\leq \int_a^b \varphi(f(s))d\mu(s)\Big/\int_a^b d\mu(s),
\end{equation}
if $\varphi$ is convex and $f$ is integrable with respect to a probability  measure $\mu$. Letting $\varphi_{z}(s)= e^{-zt},\;f(t)=t,$ and 
$$d\mu(t)= H_{q,p}^{p,0}\left(t\Big|^{(B_q,\beta_q)}_{(A_p,\alpha_p)}\right)\frac{dt}{t}.$$
Thus,
$$\int_0^\rho d\mu(t)=\prod_{i=1}^p\frac{\Gamma(\alpha_i)}{\Gamma(\beta_i)},\;\textrm{and}\;\;\int_0^\rho f(t)d\mu(t)=\frac{\prod_{i=1}^p\Gamma(\alpha_i+A_i)}{\prod_{i=1}^q\Gamma(\beta_i+B_i)},$$
and
$$\int_0^\rho \phi_z(f(t))d\mu(t)={}_p\Psi_q\Big[_{(\beta_q,B_q)}^{(\alpha_p,A_p)}\Big|z\Big].$$
This proves the lower bound asserted by Theorem \ref{T6}.  In order to demonstrate the upper bound, we will apply the converse Jensen inequality, due to Lah and Ribari\'c, which reads as follows. Set
$$A(f)=\int_m^M f(s)d\sigma(s)\Big/ \int_m^M d\sigma(s),$$
where $\sigma$ is a non-negative measure and $f$ is a continuous function. If $-\infty<m<M<\infty$ and $\varphi$ is convex on $[m,M],$ then according to \cite[Theorem 3.37]{PE}
\begin{equation}\label{CC}
(M-m)A(\varphi(f))\leq (M-A(f))\varphi(m)+(A(f)-m)\varphi(M).
\end{equation}
Setting $\varphi_z(t)=e^{-zt},\;d\sigma(t)=d\mu(t),\;f(s)=s$ and $[m,M]=[0,\rho]$, we complete the proof of the upper bound in (\ref{§§}).
\end{proof}

\begin{corollary}\label{C1ààà} Let $\lambda>0$ and under the conditions of Theorem \ref{T6}, then the following two-sided inequality holds true:
\begin{equation}\label{abel}
\frac{\psi_{0,0}\Gamma(\lambda)}{\left(1+\frac{\psi_{0,1}}{\psi_{0,0}}z\right)^\lambda}\leq{}_{p+1}\Psi_p\left[_{\;\;\;(\beta_q,B_q)}^{(\lambda,1),(\alpha_p,A_p)}\Big|-z\right]\leq \Gamma(\lambda)\left[\psi_{0,0}-\frac{\psi_{0,1}}{\rho}\left(1-\frac{1}{(1+\rho z)^\lambda}\right)\right],\;z>0.
\end{equation}
\end{corollary}
\begin{proof}Multiply inequalities (\ref{§§})  written for ${}_p\Psi_q\left[_{(b,\alpha)}^{(\alpha_p,A_p)}\Big|-zt\right]$ by $e^{-t}t^{\lambda-1}$  integrate using using the integral representation \cite[eq. 7]{P}
\begin{equation}
\int_0^\infty e^{-t}t^{\lambda-1}{}_p\Psi_q\left[_{(\beta_q,B_q)}^{(\alpha_p,A_p)}\Big|-zt\right]dt={}_{p+1}\Psi_q\left[_{(\beta_q, B_q)}^{(\lambda,1),(\alpha_p, A_p)}\Big|-z\right]
\end{equation}
and make use of the following known formula
$$\int_0^\infty t^\lambda e^{-\sigma t}dt=\frac{\Gamma(\lambda+1)}{\sigma^{\lambda+1}},$$
where  $\lambda>-1$ and $\sigma>0.$ This completes the proof of the two-sided inequality (\ref{abel}) asserted by Corollary \ref{C1ààà}.
\end{proof}

\emph{Concluding remarks:}\\

\noindent\textbf{a.} Under the hypotheses of Theorem \ref{T2}, we get
\begin{equation}\label{abel1}
\psi_{0,0} e^{-\psi_{0,1}\psi_{0,0}^{-1}z}\leq{}_{p}\Psi_p\Big[_{(\beta_q, A)}^{(\alpha_p, A)}\Big|-z\Big]\leq\psi_{0,0}-\psi_{0,1}(1-e^{-z}),\;z\in\mathbb{R}.
\end{equation}
and consequently
\begin{equation}\label{§§1}
\frac{\psi_{0,0}\Gamma(\lambda)}{\left(1+\frac{\psi_{0,1}}{\psi_{0,0}}z\right)^\lambda}\leq{}_{p+1}\Psi_p\left[_{\;\;\;(\beta_q,A)}^{(\lambda,1),(\alpha_p,A)}\Big|-z\right]\leq \Gamma(\lambda)\left[\psi_{0,0}-\psi_{0,1}\left(1-\frac{1}{(1+z)^\lambda}\right)\right],\;z>0.
\end{equation}
We note that you can proved the inequalities (\ref{abel1}) by using the integral representation (\ref{!!!!}). Obviously,  let $d\mu(u)=A^{-1}G_{p,p}^{p,0}\left(u^{1/A}\Big|^{\beta_p}_{\alpha_p}\right)\frac{du}{u}=H_{p,p}^{p,0}\left(u\Big|^{(A,\beta_p)}_{(A,\alpha_p)}\right)\frac{du}{u},\;\varphi_z(u)=e^{-zu}$ and $[m,M]=[0,1]$ and  repeating the same calculations as above with the inequalities (\ref{abel}). \\

\noindent\textbf{b.} Suppose that the hypotheses of Theorem \ref{T2} (or Corollary \ref{corollary}) are satisfied and taking in (\ref{§§1})  (resp. in (\ref{abel})) the value $A=1$ (resp. $A_p=B_p=1,\;p=q$ ) and using the identities (\ref{rrrrr}), we re-obtain the Luke's inequalities for the hypergeometric function ${}_{p+1}F_p:$ (see \cite[ Theorem 13, Eq. (4.20)]{LU})
\begin{equation}
\frac{1}{(1+\theta z)^\sigma}\leq {}_{p+1}F_p\Big[^{\;\sigma, \alpha_1,...,\alpha_p}_{\;\beta_1,...,\beta_p}\Big|-z\Big]\leq1-\theta+\frac{\theta}{(1+x)^\sigma},\;\Big(\theta=\prod_{j=1}^p\frac{\alpha_j}{\beta_j},\;z>0.\Big)
\end{equation}

\subsection{New inequalities for the generalized Mathieu's series}
The generalized Mathieu series is defined by \cite{Z}:
\begin{equation}\label{math}
S_\mu^{(\alpha,\beta)}(r;\textbf{a})=S_\mu^{(\alpha,\beta)}(r;\{a_k\}_{k=0}^\infty)=\sum_{k=1}^\infty\frac{2a_k^\beta}{(r^2+a_k^\alpha)^\mu},\;(r,\alpha,\beta,\mu>0),
\end{equation}
where it is tacitly assumed that the positive sequence  $$\textbf{a}=(a_k)_k,\;\textrm{such\;that}\;\lim_{k\longrightarrow\infty}a_k=\infty,$$
is so chosen  that the infinite series in the definition (\ref{math}) converges, that is, that the following auxiliary series:
$$\sum_{k=0}^\infty \frac{1}{a_k^{\mu\alpha-\beta}}$$
is convergent. 
\begin{theorem}\label{CC} Let $\alpha,\beta,\nu,\mu>0$ such that $\nu(\mu \alpha-\beta)>1$ and $\nu\alpha=1.$ Then the following inequalities holds true:
\begin{equation}\label{kjk}
\begin{split}
L &\leq S_{\mu}^{(\alpha,\beta)}(r;\{k^\nu\}_{k=1}^\infty)\leq R,\;r>0,
\end{split}
\end{equation}
where
$$L=2\zeta(\nu(\mu \alpha-\beta),\frac{\mu}{\nu(\mu \alpha-\beta)}r^2+1),$$
and
$$R=2\left(1-\frac{\mu}{\nu(\mu \alpha-\beta)}\right)\zeta(\nu(\mu \alpha-\beta))+\frac{2\mu}{\nu(\mu \alpha-\beta)}\zeta(\nu(\mu \alpha-\beta),r^2+1),$$
and $\zeta(s,a)$ is the Hurwitz Zeta Function defined by:
$$\zeta(s,a)=\sum_{n=0}^\infty\frac{1}{(n+a)^s},\;\Re(s)>1.$$
\end{theorem}
\begin{proof} We make use  the representation integral for the Mathieu's series \cite{KZ}, 
$$S_{\mu}^{(\alpha,\beta)}(r;\{k^\nu\}_{k=1}^\infty)=\frac{2}{\Gamma(\mu)}\int_0^\infty \frac{x^{\nu(\mu \alpha-\beta)-1}}{e^{x}-1}{}_1\Psi_1\left(^{(\mu,1)}_{(\nu(\mu \alpha-\beta),\nu\alpha)}\Big|-r^2x^{\nu\alpha}\right)dx,$$
with (\ref{§§}) and using the following formula \cite[Eq. 8, p.313]{ER}
$$\int_0^\infty\frac{x^{s-1}e^{-ax}}{1-e^{-x}}dx=\Gamma(s)\zeta(s,a),\;\Re(s)>1,\;\Re(a)>0,$$
we obtain the inequalities (\ref{kjk}) asserted by Theorem \ref{CC}.
\end{proof}
\begin{corollary}Assume that $\alpha,\beta,\nu,\mu>0$ such that $\nu(\mu \alpha-\beta)>2$ and $\nu\alpha=1.$ Then
\begin{equation}
L_1\leq S_{\mu}^{(\alpha,\beta)}(r;\{k^\nu\}_{k=1}^\infty)\leq R_1, \;r>0,
\end{equation}
where
$$L_1=\frac{2e^{-(\nu(\mu \alpha-\beta)-1)\psi\left(\frac{\mu r^2}{\nu(\mu \alpha-\beta)}+\frac{3}{2}\right)}}{\nu(\mu \alpha-\beta)-1}$$
and
$$R_1=2\left(1-\frac{\mu}{\nu(\mu \alpha-\beta)}\right)\frac{e^{(\nu(\mu \alpha-\beta)-1)\gamma}}{\nu(\mu \alpha-\beta)-1}+\frac{2\mu}{\nu(\mu \alpha-\beta)}\frac{e^{-(\nu(\mu \alpha-\beta)-1)\psi(r^2+1)}}{\nu(\mu \alpha-\beta)-1}$$
with $\gamma$ is Euler-Mascheroni constant and $\psi$ is the digamma function.
\end{corollary}
\begin{proof}The result follows from Theorem \ref{CC} combined with Theorem 3.1 \cite{NB}.
\end{proof}

\end{document}